\documentclass[11pt]{amsart}
\usepackage{amssymb}
\usepackage{amsmath}
\usepackage{fancyhdr}
\usepackage[british]{babel}
\usepackage{geometry,mathtools}
\usepackage{ulem}
\usepackage{enumitem}
\usepackage{dsfont}
\usepackage[section]{algorithm}
\usepackage{graphicx}
\usepackage[font={footnotesize}]{caption}
\usepackage[usenames,dvipsnames,table]{xcolor}
\usepackage{pstricks,tikz}
\usetikzlibrary{patterns,calc,fadings,decorations.pathreplacing}
\usepackage[
		bookmarksopen=true,
		bookmarksopenlevel=1,
		colorlinks=true,
		linkcolor=darkblue,
        linktoc=page,
		citecolor=darkblue,
]{hyperref}

\title[]{On the $(6,4)$-problem of Brown, Erd\H{o}s and S\'os}

\author[S.~Glock]{Stefan Glock}
\author[F.~Joos]{Felix Joos}
\author[J.~Kim]{Jaehoon Kim}
\author[M.~K\"uhn]{Marcus K\"uhn}
\author[L.~Lichev]{Lyuben Lichev}
\author[O.~Pikhurko]{Oleg Pikhurko}

\address[S.~Glock]{Fakult\"at f\"ur Informatik und Mathematik, Universit\"at Passau, Germany}
\email{stefan.glock@uni-passau.de}

\address[F.~Joos, M.~K\"uhn]{Institut f\"ur Informatik, Universit\"at Heidelberg, Germany}
\email{[joos, kuehn]@informatik.uni-heidelberg.de}

\address[J.~Kim]{Department of Mathematical Sciences, KAIST, South Korea}
\email{jaehoon.kim@kaist.ac.kr}

\address[L.~Lichev]{Department of Mathematics, Universit\'e Jean Monnet, Saint-Etienne, France}
\email{lyuben.lichev@univ-st-etienne.fr}

\address[O.~Pikhurko]{Mathematics Institute and DIMAP,
University of Warwick, Coventry, United Kingdom}
\email{o.pikhurko@warwick.ac.uk}

\date{\today}

\thanks{The research leading to these results was supported by Dr.~Max R\"ossler, the Walter Haefner Foundation and the ETH Z\"urich Foundation (S. Glock), the Deutsche Forschungsgemeinschaft (DFG, German Research Foundation) -- 428212407 (F. Joos and M. K\"uhn), the POSCO Science Fellowship of POSCO TJ Park Foundation (J. Kim) as well as ERC Advanced Grant 101020255 and Leverhulme Research Project Grant RPG-2018-424 (O. Pikhurko).}

\numberwithin{equation}{section}

\definecolor{darkblue}{rgb}{0,0,0.5}

\newtheorem{theorem}[algorithm]{Theorem}

\newtheorem{lemma}[algorithm]{Lemma}
\newtheorem{cor}[algorithm]{Corollary}

\theoremstyle{definition}
\newtheorem{problem}[algorithm]{Problem}
\newtheorem{conj}[algorithm]{Conjecture}

\def\lateproof#1{\removelastskip\penalty55\medskip\noindent\begin{stepenv}\end{stepenv}{\bf Proof of #1. }} 

\def\noproof{{\unskip\nobreak\hfill\penalty50\hskip2em\hbox{}\nobreak\hfill%
       $\square$\parfillskip=0pt\finalhyphendemerits=0\par}\goodbreak}
\def\endproof{\noproof\bigskip}

\newcounter{stepenv}
\newenvironment{stepenv}[1][]{\refstepcounter{stepenv}}{}

\newcounter{step}[stepenv]

\newcounter{substep}[step]
\renewcommand{\thesubstep}{\thestep.\arabic{substep}}

\newcounter{claim}[stepenv]
\newenvironment{claim}[1][]{\refstepcounter{claim}\par\medskip\noindent%
        \textit{Claim~\theclaim. #1} \itshape\rmfamily}{\medskip}

\newcommand{\cC}{\mathcal{C}}

\newcommand{\cF}{\mathcal{F}}

\newcommand{\cY}{\mathcal{Y}}

\newcommand{\bE}{\mathbb{E}}

\def\eps{{\varepsilon}}

\newcommand{\Set}[1]{\{#1\}}

\def\In{\subseteq}

\newcommand{\ex}{{\mathrm{ex}}}

\def\COMMENT#1{}
\def\TASK#1{}
\let\TASK=\footnote             
\let\COMMENT=\footnote          

\begin{document}

\begin{abstract}  
\noindent
Let $f^{(r)}(n;s,k)$ be the maximum number of edges of an $r$-uniform hypergraph on~$n$ vertices not containing a subgraph with $k$~edges and at most $s$~vertices. In 1973, Brown, Erd\H{o}s and S\'os conjectured that the limit $$\lim_{n\to \infty} n^{-2} f^{(3)}(n;k+2,k)$$ exists for all~$k$ and confirmed it for $k=2$. Recently, Glock showed this for $k=3$. We settle the next open case, $k=4$, by showing that $f^{(3)}(n;6,4)=\left(\frac{7}{36}+o(1)\right)n^2$ as $n\to\infty$. More generally, for all $k\in \{3,4\}$, $r\ge 3$ and $t\in [2,r-1]$, we compute the value of the limit $\lim_{n\to \infty} n^{-t}f^{(r)}(n;k(r-t)+t,k)$, which settles a problem of Shangguan and Tamo.
\end{abstract}

\maketitle

\section{Introduction}

For a family $\cF$ of $r$-uniform hypergraphs (or $r$-graphs for short), let $\ex(n;\cF)$ denote the maximum number of edges in an $\cF$-free $r$-graph\footnote{An $r$-graph is \textit{$\cF$-free} if it does not contain any element of $\cF$ as a subgraph.} on $n$~vertices. This is called the \textit{Tur\'an number of $\cF$}, and determining it for various families $\cF$ is one of the central topics in extremal combinatorics. In this paper, we consider the family $\cF^{(r)}(s,k)$ of all $r$-graphs with $k$ edges and at most $s$ vertices. In 1973, Brown, Erd\H{o}s and S\'os~\cite{BES:73b} introduced the function 
$$f^{(r)}(n;s,k)=\ex(n;\cF^{(r)}(s,k)).$$
Using the probabilistic method, they showed that $f^{(r)}(n;s,k)=\Omega\left(n^{(rk-s)/(k-1)}\right)$ for all $s>r\ge 2$ and $k\ge 2$. If the exponent is an integer $t$, then $s = k(r-t) + t$ and therefore every $t$-set can be contained in at most $k-1$ edges. Hence, the above lower bound has the correct order of magnitude. 
If the exponent is not an integer, then even determining the order of magnitude of $f^{(r)}(n;s,k)$ is a major open problem, which encompasses for instance the famous $(7,4)$-problem and its generalisations (see e.g.~\cite{AS:06}).
In this paper, we focus on the case when the exponent is an integer. Then, it is natural to ask if $n^{-t}f^{(r)}(n;s,k)$ tends to a limit. Indeed, in the special case $r=3$ and $t=2$, Brown, Erd\H{o}s and S\'os~\cite{BES:73b} conjectured that the limit exists.

\begin{conj}[Brown, Erd\H{o}s and S\'os~\cite{BES:73b}] \label{conj:BES}
The limit $\lim_{n\to \infty}n^{-2}f^{(3)}(n;k+2,k)$ exists for all~$k\ge 2$.
\end{conj}
They confirmed the conjecture for $k=2$, where the limit is~$1/6$ and the extremal examples are given by Steiner triple systems, which exist by the fundamental work of Kirkman~\cite{kirkman:47}. 
For $k=3$, they gave a lower bound of~$1/6$ and an upper bound of~$2/9$, and Glock~\cite{glock:19} showed that the limit is~$1/5$. All other cases of Conjecture~\ref{conj:BES} were open. In general, the best currently known bounds for Conjecture~\ref{conj:BES} are
\begin{align}
\frac{1}{6} \le \liminf_{n\to \infty} n^{-2}f^{(3)}(n;k+2,k) \le \limsup_{n\to \infty} n^{-2}f^{(3)}(n;k+2,k) \le \frac{k-1}{3k}, \label{best general}
\end{align}
where the lower bound follows from recent work on large girth approximate Steiner triple systems~\cite{BW:19,GKLO:18}, and the upper bound is obtained by averaging over vertex degrees and using the fact that $f^{(2)}(n-1; k+1, k) = \lfloor \tfrac{k-1}{k} (n-1) \rfloor$ (see~\cite{Erd65}).

One of our main contributions is to confirm the conjecture of Brown, Erd\H{o}s and S\'os~\cite{BES:73b} for $k=4$, which is the next open case. In fact, we can also determine the limit.
\begin{theorem}\label{thm:64}
$\lim_{n\to \infty} n^{-2} f^{(3)}(n; 6, 4) = \frac{7}{36}$.
\end{theorem}

By doing so, we develop a quite flexible approach that can be used to solve the Brown--Erd\H{o}s--S\'os problem for more general parameters.
Shortly after the appearance of~\cite{glock:19}, Shangguan and Tamo~\cite{ST20} extended the result of Glock to every uniformity $r\ge 4$ (and fixed $t=2$, $k=3$) by showing that
\begin{equation*}
   \lim_{n\to \infty} n^{-2} f^{(r)}(n; 3r-4, 3) = \tfrac{1}{r^2 - r - 1}.
\end{equation*}
Moreover, inspired by Conjecture~\ref{conj:BES}, they asked if the limit 
\begin{equation*}
   \lim_{n\to \infty} n^{-t} f^{(r)}(n; k(r-t)+t, k)
\end{equation*}
exists for all fixed $r$, $t$ and $k$, and if so, what its value is. We note that a useful interpretation of the term $k(r-t) + t$ is that it is the number of vertices in a $k$-edge $r$-graph such that each but the first edge shares $t$ vertices with the previous ones.


Throughout this paper, we assume that $k\ge 2$ and $t\in [r-1]$, as otherwise the problem is trivial. Furthermore, for $t=1$, one can easily show that 
\[\lim_{n\to \infty} n^{-1} f^{(r)}(n; k(r-1)+1, k) = \frac{k-1}{(k-1)(r-1)+1},\]
and the extremal examples are obtained as vertex-disjoint unions of loose trees with $k-1$ edges. Thus, we only consider the case $t\ge 2$ in the sequel. 

Observe that, for $k=2$, it follows from the famous theorem of R\"odl~\cite{rodl:85} on the existence of asymptotic Steiner systems that 
\begin{equation*}
   \lim_{n\to \infty} n^{-t} f^{(r)}(n; 2r-t, 2) = \frac{(r-t)!}{r!}
\end{equation*}
(see~\cite{ST20} for more details).

In this paper, we completely settle the problem when $k=3$. 
\begin{theorem}\label{thm:k=3}
For every $2\le t < r$, we have $\lim_{n\to \infty} n^{-t} f^{(r)}(n;3r-2t, 3) = \tfrac{2}{t!\left(2\tbinom{r}{t}-1\right)}$.
\end{theorem}
This generalises the aforementioned results of Glock~\cite{glock:19} and Shangguan and Tamo~\cite{ST20}. We remark that the upper bound was already obtained in~\cite{ST20}, where the authors also provided a lower bound of $\frac{1}{r^t-r}$ using an algebraic construction based on a matrix-property called ``strongly 3-perfect hashing''.

Finally, we also completely settle the problem for $k=4$. Recall that the case when $r=3$ is already covered by our Theorem~\ref{thm:64}, and that the following formula does not apply in this case due to a change of behaviour when $r\ge 4$ (which is the result of an optimisation problem).
\begin{theorem}\label{thm:k=4}
For all $r\ge 4$ and $t\in [2, r-1]$, we have $\lim_{n\to \infty} n^{-t} f^{(r)}(n;4r-3t, 4) = \frac{1}{t!}\binom{r}{t}^{-1}$.
\end{theorem}

A few more sporadic results are mentioned in Section~\ref{sec discussion}.






\subsection*{Notation and terminology.} For two positive integers $m, n$, we define $[n] = \Set{1,\ldots,n}$ and $[m+1,n] = [n]\setminus [m]$. 
For a set $X$ and a nonnegative integer $s$, we refer to $\binom{X}{s}$ as the set of all subsets of $X$ of size $s$.
For an $r$-graph $G$, we denote by $V(G)$ the vertex set of $G$ and by $E(G)$ its edge set. 
We often identify $G$ with its edge set, in particular, $|G|$ denotes the number of edges of $G$. Moreover, for a set $S\subseteq V(G)$, the \textit{degree of $S$} is the number of edges of $G$ that contain it. For an integer $i\in [r]$, we denote by $\Delta_i(G)$ the \textit{maximum $i$-degree of $G$} defined as the maximum of the degrees of all $i$-subsets of $V(G)$. Moreover, for simplicity we write $\Delta$ instead of $\Delta_1$. 

Usually, the uniformity $r$ is clear from the context. So, for positive integers $s$ and $k$, an \textit{$(s,k)$-configuration} is an $r$-graph with $k$ edges and at most $s$ vertices, that is, an element of $\mathcal F^{(r)}(s,k)$. Also, we say that an $r$-graph is \textit{$(s,k)$-free} to mean that it is $\mathcal F^{(r)}(s,k)$-free. 

Finally, the \textit{$t$-shadow} of an $r$-graph $F$ is the $t$-graph with vertices $V(F)$ and edges given by all $t$-subsets of the edges of $F$.

\subsection*{Overview of the proof and plan of the paper.}
We briefly sketch the main approach. For simplicity, we stick to the setting of Theorem~\ref{thm:64} since the proofs of the other theorems use similar ideas.
In order to show the existence of $n$-vertex $3$-graphs with roughly $7n^2/36$ edges and no $(6,4)$-configuration, we proceed as follows. We find a special $3$-graph $T_7$ which has seven edges and is $(6,4)$-free (see Figure~\ref{fig:64graph} for an illustration). The main idea is to pack many copies of this $3$-graph together. More precisely, we consider the $2$-shadow $J$ of $T_7$, which has $18$ edges. We then find an asymptotically optimal packing of the complete $2$-graph $K_n$ with copies of $J$. Putting a copy of $T_7$ ``on top" of each such copy of $J$ leads to roughly $|T_7|\cdot\frac{\binom{n}{2}}{|J|}\approx \frac{7}{36}n^2$ edges in the resulting $3$-graph, as desired. The challenge is to ensure that in this packing, we do not create any $(6,4)$-configurations using edges from different copies of $T_7$. For this, a crucial ingredient in our proof is a recent result on ``conflict-free hypergraph matchings''~\cite{DP22,GJKKL22}, which we introduce in Section~\ref{sec prelim}.

In order to show that the obtained lower bound is asymptotically optimal, we proceed as follows. Let $G$ be any $(6,4)$-free $n$-vertex $3$-graph. As a first step, we show that one can remove all $(4,3)$-configurations by only deleting a negligible amount of edges. For such a ``cleaned'' $3$-graph, we show that its edges can be clustered in such a way that every pair of vertices is only associated to one cluster.
An upper bound on the total number of edges can then be obtained by maximizing the ratio of edges to associated pairs over all feasible clusters. The clustering procedure is divided into two steps. Firstly, we consider the tight components of the given (cleaned) $3$-graph. All tight components must be small since otherwise one could easily find a $(6,4)$-configuration. Then, we define a merging operation which 
combines the given components into larger clusters along common pairs of vertices. The crux is to show that by the $(6,4)$-freeness, each of the obtained clusters has bounded size. This allows us to maximize the aforementioned ratio over all possible clusters (and is in fact how we found the $3$-graph $T_7$ used in the lower bound).




\section{Preliminaries}\label{sec prelim}


In this section, we introduce \textit{conflict-free hypergraph matchings}, a general tool developed recently by Glock, Joos, Kim, K\"uhn and Lichev~\cite{GJKKL22}, and independently by Delcourt and Postle~\cite{DP22}. These works were motivated in turn by earlier results of Glock, K\"uhn, Lo and Osthus~\cite{GKLO:18} as well as Bohman and Warnke~\cite{BW:19}, who proved an approximate version of an old conjecture of Erd\H{o}s~\cite{Erd73} on the existence of high-girth Steiner triple systems. Roughly speaking, a Steiner triple system has large girth if it does not contain any $(\ell+2,\ell)$-configurations for small~$\ell$. This makes the connection to our Tur\'an problem immediate.

Let $H$ be an $N$-vertex $r$-uniform hypergraph which is almost $d$-regular and has codegree $o(d)$. Then, by a celebrated theorem of Frankl and R\"odl~\cite{FR85}, and Pippenger and Spencer~\cite{PS89}, $H$ has a matching containing $(1 - o(1))N$ of the vertices in~$H$.

Now, let in addition $\cC$ be a (not necessarily uniform) hypergraph with vertex set $E(H)$. We see $\cC$ as a set of conflicts between certain edges of $H$. More precisely, we say that a matching $M$ in $H$ is \textit{$\cC$-free} if no edge of $\cC$ is a subset of $M$. Let us provide an example: given the complete graph $K_n$, let $V(H)$ be the set of all $N = \binom{n}{2}$ edges, and $E(H)$ be the set of triples of edges which form a triangle. Then, a matching in $H$ is a triangle packing of~$K_n$. Let also $\cC$ represent the set of copies of the Pasch configuration depicted in Figure~\ref{fig Pasch}. Then, a $\cC$-free matching in $H$ would be a Pasch-free triangle packing of $K_n$.

\begin{figure}[ht]
\centering
\definecolor{ffffww}{rgb}{1,1,0.4}
\definecolor{ffqqqq}{rgb}{1,0,0}
\definecolor{yqyqyq}{rgb}{0.5019607843137255,0.5019607843137255,0.5019607843137255}
\definecolor{rvwvcq}{rgb}{0.08235294117647059,0.396078431372549,0.7529411764705882}
\begin{tikzpicture}[line cap=round,line join=round,x=1cm,y=1cm]
\clip(-8,0.9) rectangle (0,3.1);
\fill[line width=0.2pt,color=black,fill=yqyqyq,fill opacity=0.3] (-6,1) -- (-4.5,2) -- (-4,3) -- cycle;
\fill[line width=0.2pt,color=black,fill=ffqqqq,fill opacity=0.3] (-4,3) -- (-3.5,2) -- (-2,1) -- cycle;
\fill[line width=0.2pt,color=black,fill=rvwvcq,fill opacity=0.3] (-6,1) -- (-4,1) -- (-3.5,2) -- cycle;
\fill[line width=0.2pt,color=black,fill=ffffww,fill opacity=0.3] (-4,1) -- (-2,1) -- (-4.5,2) -- cycle;

\draw [line width=0.2pt,color=black] (-6,1)-- (-4.5,2);
\draw [line width=0.2pt,color=black] (-4.5,2)-- (-4,3);
\draw [line width=0.2pt,color=black] (-4,3)-- (-6,1);
\draw [line width=0.2pt,color=black] (-4,3)-- (-3.5,2);
\draw [line width=0.2pt,color=black] (-3.5,2)-- (-2,1);
\draw [line width=0.2pt,color=black] (-2,1)-- (-4,3);
\draw [line width=0.2pt,color=black] (-6,1)-- (-4,1);
\draw [line width=0.2pt,color=black] (-4,1)-- (-3.5,2);
\draw [line width=0.2pt,color=black] (-3.5,2)-- (-6,1);
\draw [line width=0.2pt,color=black] (-4,1)-- (-2,1);
\draw [line width=0.2pt,color=black] (-2,1)-- (-4.5,2);
\draw [line width=0.2pt,color=black] (-4.5,2)-- (-4,1);
\begin{scriptsize}
\draw [fill=black] (-6,1) circle (1.5pt);
\draw [fill=black] (-4.5,2) circle (1.5pt);
\draw [fill=black] (-4,3) circle (1.5pt);
\draw [fill=black] (-3.5,2) circle (1.5pt);
\draw [fill=black] (-2,1) circle (1.5pt);
\draw [fill=black] (-4,1) circle (1.5pt);
\end{scriptsize}
\end{tikzpicture}
\caption{The Pasch configuration is the above 2-regular 3-graph on six vertices and four edges.}
\label{fig Pasch}
\end{figure}
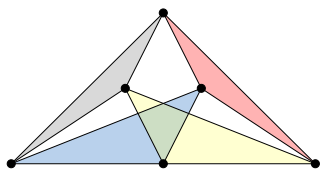


Let $H$ be as above. Assume that all edges of $\cC$ have size at most $\ell=O(1)$ and at least 2. For every $j\in [2,\ell]$, denote by $\cC^{(j)}$ the spanning subgraph of $\cC$ containing the edges of $\cC$ of size~$j$. We also assume that, for every $j\in [2,\ell]$, the maximum degree of $\cC^{(j)}$ is $O(d^{j-1})$.

For instance, in the Pasch configuration example described above, every triangle in $K_n$ (which is an edge in $H$ and therefore a vertex in $\cC$) is contained in $6\tbinom{n-3}{3}=O(d^3)$ Pasch configurations. Here, $d=n-2$ is the degree of a vertex in $H$. Moreover, edges in $\cC$ have size $4$ since the Pasch configuration consists of four triples, so indeed $\Delta(\cC^{(4)}) = O(d^3)$.

Note that the required relation $\Delta(\cC^{(j)})=O(d^{j-1})$ is natural: it implies that, in total, there are $O(Nd^j)$ conflicts of size exactly $j$. Using a probabilistic deletion method, one can select every edge of $H$ with probability $\eps/d$ for some small $\eps$ to form a (random) set $S\subseteq E(H)$. 
Then, $\bE[|S|] = \Omega(\eps N)$. Now, delete from $S$ all overlapping edges (i.e. all $e\in S$ such that, for some $f\in S$, we have $e\cap f \neq \emptyset$) as well as edges forming conflicts (i.e.\ all $e\in S$ such that, for some $C\in \cC$, we have $e\in C \text{ and }C\subseteq S$). Since the expected number of such edges is 
$$O\bigg(\frac{\eps}{d}\cdot \frac{Nd}{r}\cdot \bigg(rd\cdot \frac{\eps}{d}\bigg)\bigg) + O\bigg(\sum_{j=2}^{\ell} N d^j\cdot  \bigg(\frac{\eps}{d}\bigg)^j\bigg) = O(\eps^2 N),$$
there are on average $\Omega(\eps N)$ edges that remain after all deletions, and these form a $\cC$-free matching.

In order to get not only a linear size matching but an almost-perfect one, we need a few further codegree assumptions for $\cC$, customary to the R\"odl nibble.

\begin{theorem}[Theorem 1.3 from~\cite{GJKKL22}]\label{thm:matching}
For all~$r,\ell\geq 2$, there exists~$\eps_0>0$ such that for all~$\eps\in(0,\eps_0)$, there exists~$d_0$ such that the following holds for all~$d\geq d_0$. Let $H$ be an $r$-graph with $|V(H)|\le \exp(d^{\eps^3})$ such that every vertex is contained in $(1\pm d^{-\eps})d$ edges and~$\Delta_2(H)\le d^{1-\eps}$.

Let $\cC$ be a hypergraph with $V(\cC)=E(H)$ such that every $C\in E(\cC)$ satisfies $2\le |C|\le \ell$, and the following conditions hold.
\begin{enumerate}
	\item[$(\mathrm{C}1)$]\label{cn C1} $\Delta(\cC^{(j)})\le \ell d^{j-1}$ for all~$2\le j\le \ell$;
	\item[$(\mathrm{C}2)$]\label{cn C2} $\Delta_{j'}(\cC^{(j)})\le d^{j-j'-\eps}$ for all $2\le j'<j\le \ell$;
	\item[$(\mathrm{C}3)$]\label{cn C3} $|\{ f\in H: \{e,f\}\in \cC \text{ and } v\in f \}|\leq d^{1-\eps}$ for all~$e\in E(H)$ and~$v\in V(H)$;
	\item[$(\mathrm{C}4)$]\label{cn C4} $|\{g\in H: \{e,g\}, \{f,g\}\in \cC\}|\leq d^{1-\eps}$ for all disjoint~$e,f\in H$.
\end{enumerate}
Then, there exists a $\cC$-free matching $M$ in $H$ which covers all but $d^{-\eps^3} |V(H)|$ vertices of~$H$.
\end{theorem}

\noindent
We remark that a version of Theorem~\ref{thm:matching} appears as Corollary~1.17 in~\cite{DP22}.

\section{Lower bounds}\label{sec LB}

To prove the lower bounds in Theorems~\ref{thm:64}, \ref{thm:k=3} and~\ref{thm:k=4}, we pack the (edges of the) complete $t$-graph $K^t_n$ with disjoint $t$-shadows of carefully chosen $r$-graphs of constant size. To start with, we explain how to deduce lower bounds on $\liminf_{n\to \infty} n^{-t}f^{(r)}(n;k(r-t)+t,k)$ in terms of certain $(k(r-t)+t,k)$-free $r$-graphs.

Given an $r$-graph $F$ and a $t$-graph $J$, we say that $J$ is a \textit{supporting $t$-graph} of $F$ if $V(J)=V(F)$ and for every edge $e\in E(F)$, all the $t$-subsets of $e$ are edges in~$J$, that is, $J$ contains the $t$-shadow of~$F$. Also, for $F$ and $J$ as above, we define the \textit{non-edge girth} of $(F,J)$ to be the smallest $g\ge 1$ for which there exists an $(g(r-t)+t, g)$-configuration in $F$ which contains (as a vertex subset) a non-edge of~$J$. If no such $g$ exists, we set the non-edge girth to be infinity.

\begin{theorem}\label{thm:pack general}
Fix $k\ge 2$, $r\ge 3$ and $t\in [2,r-1]$. Let $F$ be an $r$-graph which is $(k(r-t)+t,k)$-free and $(\ell(r-t)+t-1,\ell)$-free for all $\ell\in [2,k-1]$. Let $J$ be a supporting $t$-graph of $F$ such that the non-edge girth of $(F,J)$ is greater than $k/2$. Then, 
$$\liminf_{n\to \infty} n^{-t}f^{(r)}(n;k(r-t)+t,k) \ge \frac{|F|}{t!\,|J|}.$$
\end{theorem}


Before proving this result, let us see some applications.

\begin{cor}\label{cor:complete pack}
Fix $k\ge 2$, $r\ge 3$ and $t\in [2,r-1]$. Also, fix a $(k(r-t)+t,k)$-free $r$-graph $F$ on $m$ vertices which is also $(\ell(r-t)+t-1,\ell)$-free for all $\ell\in [2,k-1]$. Then, 
$$\liminf_{n\to \infty} n^{-t}f^{(r)}(n;k(r-t)+t,k) \ge \frac{|F|}{m^t}.$$
\end{cor}

\begin{proof}
Take $J$ to be the complete $t$-graph on $V(F)$. Then, $t!\,|J| < m^t$ and the non-edge girth is infinite, so Theorem~\ref{thm:pack general} gives the required conclusion.
\end{proof}

Note that Corollary~\ref{cor:complete pack} applied with $(t,r) = (2,3)$ ensures that a positive answer to the following problem would imply Conjecture~\ref{conj:BES}.

\begin{problem}\label{prob:clean}
Let $k\ge 3$. Given a $(k+2,k)$-free $3$-graph $F$ on $m$ vertices, can one remove $o(m^2)$ edges and make it $(\ell+1,\ell)$-free for all $\ell\in [2,k-1]$?
\end{problem}

Indeed, fix $k\ge 3$ and define $\pi_k=\limsup_{n\to \infty} n^{-2}f^{(3)}(n;k+2,k)$. Also, fix $\eps>0$. Then, by the definition of $\limsup$ there exist an arbitrarily large $m$ and a $(k+2,k)$-free $3$-graph $F'$ on $m$ vertices such that $|F'|\ge (\pi_k-\eps/2) m^2$. If one can delete $\eps m^2/2$ edges from $F'$ to obtain a 3-graph $F$ which is $(\ell+1,\ell)$-free for all $\ell\in [2,k-1]$, then applying Corollary~\ref{cor:complete pack} with $F$ shows that 
$$\liminf_{n\to \infty} n^{-2}f^{(3)}(n;k+2,k) \ge \frac{|F|}{m^2}\ge \pi_k-\eps.$$ 
Since this holds for all $\eps>0$, we obtain that $\liminf_{n\to \infty} n^{-2}f^{(3)}(n;k+2,k)=\pi_k$. Note that an analogous reduction can be done more generally for $n^{-t} f^{(r)}(n; k(r-t)+t, k)$.

As another application of Theorem~\ref{thm:pack general}, we can derive the lower bounds stated in the introduction.

\begin{proof}[Proof of the lower bounds]
We start with Theorem~\ref{thm:k=3}: let $F$ be an $r$-graph consisting of two $r$-edges intersecting in $t$ vertices, and let $J$ be its $t$-shadow. Note that $F$ is $(2(r-t)+t-1, 2)$-free and the non-edge girth of $(F,J)$ is $2 > 3/2$. Hence, by Theorem~\ref{thm:pack general} we have that
\begin{equation*}
    \liminf_{n\to \infty} n^{-t} f^{(r)}(n; 3r-2t, 3) \ge \frac{|F|}{t!\, |J|} = \frac{2}{t!\left(2\tbinom{r}{t}-1\right)}.
\end{equation*}

The lower bound in Theorem~\ref{thm:k=4} follows by applying Theorem~\ref{thm:pack general} for $F$ being a single edge and $J = \binom{V(F)}{t}$ (in which case the non-edge girth of $(F,J)$ is infinity). 

It remains to prove the lower bound in Theorem~\ref{thm:64}. Define the $3$-graph $T_7$ with vertex set $\bigcup_{i=1}^3 \{x_i,a_i,b_i\}$ and edge set 
$$\{x_1x_2x_3, a_1b_1x_2, a_1b_1x_3, a_2b_2x_1, a_2b_2x_3, a_3b_3x_1, a_3b_3x_2\},$$ 
whose supporting graph $J$ is simply its 2-shadow (see Figure~\ref{fig:64graph}); in particular, $|J|=18$. By construction $T_7$ is $(6,4)$-free and $(4,3)$-free (and, vacuously, $(3,2)$-free). Moreover, its non-edge girth is at least 3: indeed, the only $(4,2)$-configurations have vertex sets $\{a_3, b_3, x_1, x_2\}$, $\{a_1, b_1, x_2, x_3\}$ and $\{a_2, b_2, x_1, x_2\}$, which all induce copies of $K_4$ in $J$.\footnote{In fact, the non-edge girth of $T_7$ is exactly 3 because $T_7$ contains the $(5,3)$-configuration $\{x_1x_2x_3, a_2b_2x_1, a_2b_2x_3\}$ and $b_2x_2\notin J$.} Hence, applying Theorem~\ref{thm:pack general} directly yields the lower bound of Theorem~\ref{thm:64}.
\end{proof}


\begin{figure}[ht]
	\begin{center}
		\begin{tikzpicture}
			\coordinate (x) at (90:1);
			\coordinate (y) at (210:1);
			\coordinate (z) at (-30:1);
			\coordinate (a) at ($ (y) + (-90:2) $);
			\coordinate (b) at ($ (z) + (-90:2) $);
			\coordinate (c) at ($ (x) + (150:2) $);
			\coordinate (d) at ($ (y) + (150:2) $);
			\coordinate (e) at ($ (x) + (30:2) $);
			\coordinate (f) at ($ (z) + (30:2) $);
			
			\draw [fill=black] (-1.2,-0.6) node {$x_1$};
			\draw [fill=black] (1.2,-0.6) node {$x_2$};
			\draw [fill=black] (0,1.25) node {$x_3$};
			\draw [fill=black] (-1.2,-2.6) node {$a_3$};
			\draw [fill=black] (1.2,-2.6) node {$b_3$};
			\draw [fill=black] (-2.05,2) node {$a_2$};
			\draw [fill=black] (-2.9,0.5) node {$b_2$};
			\draw [fill=black] (2.05,2) node {$b_1$};
			\draw [fill=black] (2.9,0.5) node {$a_1$};
			
			\begin{scope}[opacity=0.6]
				\draw[fill,orange!60,rounded corners] (y)--(c)--(d)--cycle;
				\draw[fill,Cyan!60,rounded corners] (x)--(c)--(d)--cycle;
				\draw[fill,orange!60,rounded corners] (x)--(e)--(f)--cycle;
				\draw[fill,Cyan!60,rounded corners] (z)--(e)--(f)--cycle;
				\draw[fill,orange!60,rounded corners] (y)--(a)--(b)--cycle;
				\draw[fill,Cyan!60,rounded corners] (z)--(a)--(b)--cycle;
				\draw[fill,orange!60,rounded corners] (x)--(y)--(z)--cycle;
			\end{scope}
			
			\draw (x)--(y)--(z)--cycle;
			\draw (y)--(a)--(z)--(b)--cycle;
			\draw (a)--(b);
			\draw (x)--(c)--(y)--(d)--cycle;
			\draw (c)--(d);
			\draw (x)--(e)--(z)--(f)--cycle;
			\draw (e)--(f);
			
			
			\draw[fill] (x) circle (1.5pt);
			\draw[fill] (y) circle (1.5pt);
			\draw[fill] (z) circle (1.5pt);
			\draw[fill] (a) circle (1.5pt);
			\draw[fill] (b) circle (1.5pt);
			\draw[fill] (c) circle (1.5pt);
			\draw[fill] (d) circle (1.5pt);
			\draw[fill] (e) circle (1.5pt);
			\draw[fill] (f) circle (1.5pt);

		\end{tikzpicture}
	\end{center}
	\caption{The $3$-graph $T_7$ has seven edges and its 2-shadow contains $18$ edges. It is easy to check that this graph is $(6,4)$-free, and also $(4,3)$-free. Moreover, its non-edge girth is $3$.}\label{fig:64graph}
\end{figure}
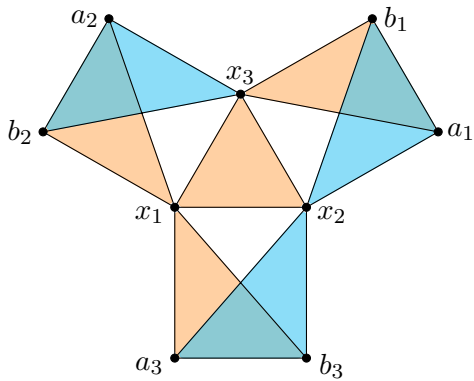

\lateproof{Theorem~\ref{thm:pack general}}
Fix $k$, $r$, $t$, $F$ and $J$ as in the statement of the theorem. Let $m=|V(F)|$ and let $\eps>0$ be arbitrarily small. In the sequel, we allow the constants in our $O(\cdot)$ notations to depend on $k$, $r$, $t$ and $\eps$. We will show that, for all sufficiently large $n$, there is a $(k(r-t)+t,k)$-free $r$-graph with vertex set $[n]$ and at least $(1-2\eps)\frac{|F|}{|J|}\binom{n}{t}$ edges, which clearly implies the result.

\subsection*{\texorpdfstring{Step 1. Defining the auxiliary hypergraph $H$.}{}}

As a first step, we $2$-colour the edges of the complete $t$-graph with vertex set $[n]$. Colour every edge (independently of all others) red with probability $\eps/2$ and blue otherwise. Given a copy $J'$ of $J$ with $V(J')\subseteq [n]$, call it \textit{admissible} if every edge of $J'$ is coloured blue and every non-edge in $\tbinom{V(J')}{t}$ is coloured red.

Now, let $H$ be a hypergraph with vertices, corresponding to the blue edges in $\tbinom{[n]}{t}$, and edges, corresponding to the admissible copies of $J$. Hence, $H$ is $|J|$-uniform. Moreover, by standard concentration arguments $H$ is with high probability approximately regular with all degrees $(1+O(d^{-\eps}))d$, where $d=c n^{m-t}$ for some positive constant $c$. In the remainder of the proof we assume that this event holds.

For every such admissible copy of $J$, we add a copy of $F$ (supported by this $J$) on top (if there are several ways to do this, pick one arbitrarily). In the sequel, we identify the edges of $H$ with these copies of~$F$; thus, several such copies form a matching if and only if their associated supporting graphs are edge-disjoint.

\subsection*{\texorpdfstring{Step 2. Defining the conflict hypergraph $\cC$.}{}}

Our main goal is to make sure that the final packing is $(k(r-t)+t,k)$-free, and hence it is natural to define $\cC_0$ as the set of all collections of copies of $F$ which form a matching in $H$ and whose union contains a $(k(r-t)+t,k)$-configuration.
Ultimately, we want to find a matching in $H$ which is $\cC_0$-free.
However, in order to verify the conditions of Theorem~\ref{thm:matching}, it will be appropriate to see $\cC_0$ as part of a bigger conflict system.

More precisely, let $\cC_1$ be the (minimal) set of all collections of copies of $F$ which form a matching in $H$ and whose union contains an $(\ell(r-t)+t,\ell)$-configuration $S$ for some $\ell\in [2,k]$ such that $S$ contains edges from at least two of the copies. Clearly $\cC_0\In \cC_1$ since no copy of $F$ contains a $(\ell(r-t)+t,\ell)$-configuration on its own. Moreover, let $\cC_2$ be the set of all pairs of copies of $F$ which intersect in at least $t+1$ vertices. (We forbid intersections of size at least $t+1$ to facilitate the verification of condition~(C4) in Theorem~\ref{thm:matching}.) Let $\cC' = \cC_1\cup \cC_2$. Finally, let $\cC$ be the set of all inclusion-wise minimal elements of~$\cC'$. Observe that if a matching in $H$ is $\cC$-free, then it is in particular $\cC_0$-free, as we need.
	
To begin with, all the conflicts in $\cC$ have size between $2$ and $k$ (so we may choose any integer larger than or equal to $k$ to play the role of $\ell$ in Theorem~\ref{thm:matching}). Below we check conditions (C1) to (C4) with the help of the following two claims.

Claim~\ref{claim:degree prep} will allow us to control the degrees of the conflict hypergraph. Here, the assumption that $F$ is $(\ell(r-t)+t-1,\ell)$-free for all $\ell\in [2,k-1]$ is crucial.

\begin{claim}\label{claim:degree prep}
Fix $j\ge 3$ and suppose that $\Set{F_1,\dots,F_j}\in \cC^{(j)}$. Then,

\[\bigg|\bigg(\bigcup_{i=2}^{j} V(F_i)\bigg)\setminus V(F_1)\bigg| \le (j-1)(m-t),\]
and for every $j'\in [2, j-1]$,
$$\bigg|\bigg(\bigcup_{i=j'+1}^{j} V(F_i)\bigg)\setminus \bigg(\bigcup_{i=1}^{j'} V(F_i)\bigg)\bigg| \le (j-j')(m-t)-1.$$
\end{claim}

\begin{proof}[Proof of Claim~\ref{claim:degree prep}]
The statement follows from the minimality of the conflicts and the fact that~$F$ is~$(\ell(r-t)+t-1,\ell)$-free for all~$\ell\in[2,k]$.

In more detail, we argue as follows.
Since~$\Set{F_1,\ldots,F_j}\in\cC^{(j)}$ and $j\geq 3$, there exists an~$(\ell(r-t)+t,\ell)$-configuration~$S\subseteq \bigcup_{i=1}^j F_i$ for some~$\ell\in[2,k]$.
For~$i\in[j]$, let~$S_i=S\cap F_i$ and $\ell_i=|S_i|$, and let~$S_{\leq i}= \bigcup_{i'=1}^{i} S_{i'}$ and~$\ell_{\leq i}=|S_{\leq i}|$.
Note that~$\ell_i\geq 1$ for all~$i\in[j]$ by minimality of the conflicts, and hence also~$\ell_{\leq i}\geq 1$.
We have
\begin{equation}\label{equation: split difference bound}
\begin{aligned}
\Bigl|\Big(\bigcup_{i=j'+1}^{j} V(F_i)\Big)\setminus \bigcup_{i=1}^{j'} V(F_i)\Bigr|
&\le\Bigl|\Big(V(S)\setminus V(S_{\leq j'})\Big)\cup\bigcup_{i=j'+1}^{j} \Big(V(F_i)\setminus V(S_i)\Big)\Bigr|\\
&\leq |V(S)|-|V(S_{\leq j'})|+\sum_{i=j'+1}^{j} |V(F_i)\setminus V(S_i)|\\
&\leq \ell(r-t)+t-|V(S_{\leq j'})|+(j-j')m-\sum_{i=j'+1}^{j}|V(S_i)|.
\end{aligned}
\end{equation}

Now, we observe that 
\begin{equation}\label{equation: initial configuration part bound}
|V(S_{\leq 1})|\geq \ell_{\leq 1}(r-t)+t.
\end{equation}
Indeed, if~$\ell_{\leq 1}= 1$, then~\eqref{equation: initial configuration part bound} trivially holds, and if~$\ell_{\leq 1}\geq 2$, then~\eqref{equation: initial configuration part bound} follows by the~$(\ell_1(r-t)+t-1,\ell_1)$-freeness of~$F_1$.
Moreover, we show that for all $j'\in [2,j-1]$,
\begin{equation}\label{eq:j'ge2}
|V(S_{\leq j'})|\geq \ell_{\leq j'}(r-t)+t+1.
\end{equation}
This is implied by the minimality of the conflicts: indeed, for all $j'\in [2,j-1]$, $\bigcup_{i=1}^{j'} F_i$ contains no forbidden configuration.

We also show that for every~$i\in[j'+1,j]$,
\begin{equation}\label{equation: subsequent configuration parts bound}
|V(S_i)|\geq \ell_i(r-t)+t.
\end{equation}
Indeed, if~$\ell_i=1$, then~\eqref{equation: subsequent configuration parts bound} trivially holds, and it is otherwise implied by the~$(\ell_i(r-t)+t-1,\ell_i)$-freeness of~$F_i$.

Finally, since~$\ell_{\leq j'}+\sum_{i=j'+1}^j \ell_i\geq \ell$, combining~\eqref{equation: split difference bound},~\eqref{equation: initial configuration part bound} and~\eqref{equation: subsequent configuration parts bound} yields the desired bound.
\color{black}
\end{proof}

\begin{claim}\label{claim:girth}
If $\Set{F_1,F_2}\in \cC^{(2)}$, then $|V(F_1)\cap V(F_2)|\ge t+1$.
\end{claim}
\begin{proof}[Proof of Claim~\ref{claim:girth}]
We show that for all~$F_1,F_2\in H$, we have~$\Set{F_1,F_2}\notin\cC_1$, which follows from the fact that the non-edge girth of~$F$ is greater than~$k/2$.

Let us turn to the details.
Fix disjoint edges~$F_1,F_2\in H$ and a subgraph~$S\subseteq F_1\cup F_2$ satisfying~$2\leq \ell=|S|\leq k$.
We show that~$|V(S)|\geq \ell(r-t)+t+1$, which in turn implies that~$\Set{F_1,F_2}\notin\cC_1$ and the claim.
Recall that for both~$i\in\Set{1,2}$, by Step 1, there is a copy $J_i$ of $J$ supporting~$F_i$, and let~$S_i=S\cap F_i$ and~$\ell_i=|S_i|$.
Since~$\ell\leq k$, we have~$|S_i|\leq k/2$ for some~$i\in\Set{1,2}$. This together with the fact that the non-edge girth of~$(F,J)$ is larger than~$k/2$ implies~$\binom{V(S_i)}{t}\subseteq J_i\subseteq V(H)$.
Since~$F_1$ and~$F_2$ are disjoint as edges of~$H$, and since~$J_1$ and~$J_2$ are both admissible copies of~$J$, this implies
\begin{equation}\label{equation: intersection bound}
|V(S_1)\cap V(S_2)|\leq t-1.
\end{equation}
(Indeed, if $|V(S_1)\cap V(S_2)|$ was at least $t$, every $t$-subset of this intersection would have belonged to $V(F_1)\cap V(F_2)$, which would contradict the fact that $F_1, F_2$ are disjoint.)
Moreover, for~$i\in\Set{1,2}$, since $\ell_i\le k-1$, by the~$(\ell_i(r-t)+t-1,\ell_i)$-freeness of~$F_i$ we have
\begin{equation}\label{equation: part bound}
|V(S_i)|\geq \ell_i(r-t)+t.
\end{equation}
Since~$\ell_1+\ell_2\geq \ell$, we combine~\eqref{equation: intersection bound} and~\eqref{equation: part bound} to conclude that
\begin{equation*}
|V(S)|
=|V(S_1)|+|V(S_2)|-|V(S_1)\cap V(S_2)|
\geq \ell(r-t) +t+1,
\end{equation*}
which completes the proof.
\color{black}
\end{proof}

We are now ready to verify the conditions of Theorem~\ref{thm:matching}. Recall that $d=\Theta(n^{m-t})$, and that every vertex of $H$ has degree $(1+O(d^{-\eps}))d$. Also, as every two distinct $t$-sets span at least $t+1$ vertices, it holds that $\Delta_2(H)\le O(n^{m-t-1})\le d^{1-\eps}$. Moreover, $|V(H)|\le n^t\le \exp(d^{\eps^3})$.

\subsection*{\texorpdfstring{Step 3. Degrees of the conflict hypergraph: verifying~(C1).}{}}
	
Let us first check the degrees in~$\cC$. Fix copies~$(F_i)_{i=1}^j$ of $F$ and suppose that $\Set{F_1,\dots,F_j}\in \cC^{(j)}$. We claim that the copies $F_2,\dots,F_j$ contain at most $(j-1)(m-t)$ vertices outside $V(F_1)$. If $j=2$, by Claim~\ref{claim:girth} the copy $F_2$ contains at most $m-t-1<m-t$ vertices outside~$F_1$. If $j\ge 3$, by Claim~\ref{claim:degree prep} the copies $F_2,\dots,F_j$ contain at most $(j-1)(m-t)$ vertices outside $V(F_1)$. Hence, the number of choices of a $j$-set $\Set{F_1,\dots,F_j}\in \cC^{(j)}$ is $O(n^{(j-1)(m-t)})$. Since $d=\Theta(n^{m-t})$, we obtain that $\Delta(\cC^{(j)})=O(d^{j-1})$, as required in~(C1).
	
\subsection*{\texorpdfstring{Step 4. Codegrees of the conflict hypergraph: verifying~(C2).
}{}}
	
Now, consider $2\le j'<j$ and copies $F_1,\dots,F_{j'}$ of $F$. Using Claim~\ref{claim:degree prep}, the number of sets $\Set{F_{j'+1},\dots,F_j}$ such that $\Set{F_{1},\dots,F_j}\in \cC^{(j)}$ is $O(n^{(j-j')(m-t)-1})=O(d^{j-j'-(m-t)^{-1}})$.

\subsection*{\texorpdfstring{Step 5. Verifying~(C3) and~(C4).}{}}

By Claim~\ref{claim:girth}, $\Set{F_1,F_2}\in \cC^{(2)}$ implies that $F_1,F_2$ intersect in at least $t+1$ vertices. Thus, for every fixed copy $F_1$, there are at most $O(n^{m-t-1}) = O(d^{1-(m-t)^{-1}})$ choices for $F_2$ that form a conflict together with $F_1$, which yields both~(C3) and~(C4) (note that fixing a vertex $v$ in (C3) or a third copy $F_3$ of $F$ such that $\{F_1, F_2\}, \{F_3, F_2\}\in \cC$ in (C4) is not even needed here).\\

\subsection*{\texorpdfstring{Step 6. Wrapping up.}{}}

Finally, we apply Theorem~\ref{thm:matching} to obtain a $\cC$-free matching $M$ in $H$ which covers all but $o(n^t)$ vertices of~$H$. Each edge in $M$ corresponds to a copy of $F$ and, by definition of $\cC$, the union of all these copies gives us an $r$-graph which is $(k(r-t)+t,k)$-free. Finally, the number of blue edges we started with is 
at least $(1-\eps)\tbinom{n}{t}$, and since we get an almost perfect matching, we know that at least $(1-2\eps)\tbinom{n}{t}$ of all edges are covered by $M$. As the corresponding copies of $J$ are edge-disjoint, this means that
$|M|\ge \tfrac{1-2\eps}{|J|}\tbinom{n}{t}$, and since every copy of $F$ contains $|F|$ edges, the desired bound follows.
\endproof

\section{Upper bounds}\label{sec UB}

In this section, we are interested in hypergraphs of arbitrary uniformity $r\ge 3$ that are $\cF^{(r)}(k(r-t)+t, k)$-free for some integers $k\in \{3,4\}$ and $t\in [2, r-1]$, that is, that do not contain subgraphs with $k$ edges spanning at most $k(r-t)+t$ vertices. The fact that the upper bound on $\limsup_{n\to \infty} n^{-t} f^{(r)}(n; 3r-2t, 3)$ stated in Theorem~\ref{thm:k=3} appears as part of Theorem~6 in~\cite{ST20} allows us to concentrate on the case $k=4$. We treat both Theorems~\ref{thm:64} and~\ref{thm:k=4} in a unified way.

The proof proceeds in several steps. We start by getting rid of configurations that are ``too dense'' by deleting a small number of edges.\\

\noindent
\textbf{Step 1: Deleting dense configurations.} We employ the following lemma, which solves the analogue of Problem~\ref{prob:clean} for~$k=4$ and arbitrary uniformities~$r$ and~$t$.
\begin{lemma}\label{lem:clean k=4}
	Let~$r\geq 3$ and~$t\in [2, r-1]$.
	Suppose~$G$ is an~$r$-graph on~$n$ vertices that is~$(4r-3t,4)$-free.
	Then, there exists a spanning $r$-graph~$G'\subseteq G$ that is~$(3r-2t-1,3)$-free,~$(2r-t-1,2)$-free and contains $|G|-O(n^{t-1})$ edges.
\end{lemma}
\begin{proof}
We show that the deletion of all edges in $(3r-2t-1, 3)$-configurations, and then in $(2r-t-1, 2)$-configurations, leads to an overall difference of $O(n^{t-1})$ edges. 
	
First, consider a $(3r-2t-1, 3)$-configuration $\{e_1, e_2, e_3\}$. Note that every edge $e\in G\setminus \{e_1, e_2, e_3\}$ satisfies $|e\cap (e_1\cup e_2\cup e_3)|\le t-2$ since otherwise $\{e, e_1, e_2, e_3\}$ would form a $(4r-3t, 4)$-configuration. In particular, every set of $t-1$ vertices is contained in at most three edges in $(3r-2t-1, 3)$-configurations, and therefore there are $O(\tbinom{n}{t-1}) = O(n^{t-1})$ edges in such configurations. Let us delete all such edges.
	
Now, we turn to $(2r-t-1, 2)$-configurations. First of all, note that an edge $e$ cannot participate in two $(2r-t-1, 2)$-configurations $\{e,f\}$ and $\{e,g\}$ since otherwise $\{e,f,g\}$ would span at most $r + 2(r-t-1) = 3r - 2t -2$ vertices. Also, every set of $t-1$ vertices can be contained in at most two edges in $(2r-t-1, 2)$-configurations because otherwise it would be contained in two different (and hence edge-disjoint) $(2r-t-1, 2)$-configurations and therefore in a $(4r-3t-1, 4)$-configuration as well. Thus, in total there are $O(\tbinom{n}{t-1}) = O(n^{t-1})$ edges in $(2r-t-1, 2)$-configurations. By deleting all such edges, we obtain an~$r$-graph~$G'$ as desired.
\end{proof}
By Lemma~\ref{lem:clean k=4}, it suffices to obtain the required upper bounds on the numbers of edges in~$(4r-3t, 4)$-free~$r$-graphs only for those that are also~$(3r-2t-1,3)$-free and~$(2r-t-1,2)$-free.
For the subsequent steps, fix any such~$r$-graph $G$ on $n$ vertices.

\vspace{1em}

In the next step, we ``cluster'' the edges of $G$ into larger blocks in such a way that every $t$-set of vertices is associated with at most one block. This is reminiscent of the connection between the $r$-graph $F$ and its supporting $t$-graph $J$ which was central in Section~\ref{sec LB}.\\

\noindent
\textbf{Step 2: Merging.} To begin with, we iteratively combine the edges of $G$ into components as follows: starting from $|G|$ components (one for every edge in $G$), at every step we look for a pair of edges $e,f\in G$ in different components such that $|e\cap f|= t$, and if there is such a pair, merge the components of $e$ and $f$. Then, we end up with a partition of $G$ into \textit{$t$-tight components}. Note that every $t$-tight component contains at most three edges since a larger component would contain a $(4r-3t,4)$-configuration.

Now, we do the following additional mergings. For a set of $t$ vertices $S$ and an edge $e\in G$, we say that $e$ \textit{covers} $S$ if $S\subseteq e$. A \textit{diamond} is a set of two edges sharing $t$ vertices. We say that a component \textit{claims} a set $S$ of $t$ vertices if it contains a diamond $\{e_1, e_2\}$ such that $S\subseteq e_1\cup e_2$ but $S$ is not contained in any of $e_1$ and $e_2$. Iteratively and as long as possible, we merge the components such that one component claims a $t$-set covered by an edge in another component and denote by $\cY$ the final collection of components. Let us call the elements of $\cY$ \textit{clusters}.

\begin{lemma}
The collection $\cY$ consists of the $t$-tight components with two or three edges and the clusters, obtained from a single edge $e$ by attaching a number $i$ between $0$ and $\tbinom{r}{t}$ of diamonds claiming distinct $t$-subsets of $e$.
\end{lemma}
\begin{proof}
Note that a diamond in a $t$-tight component with three edges cannot claim a $t$-set covered by another edge, as otherwise it would contain a $(4r-3t, 4)$-configuration. Moreover, for the same reason, every diamond may claim a $t$-set of vertices only if this set is covered by at most one edge, and this edge must be the only edge in its $t$-tight component. At the same time, one $t$-set cannot be claimed by two different diamonds. Therefore, a cluster not containing a $t$-tight component on one edge is itself a $t$-tight component on two or three edges. Finally, a cluster containing a $t$-tight component on one edge, say $e$, is the union of $e$ and diamonds claiming distinct $t$-subsets of $e$.
\end{proof}

\noindent
\textbf{Step 3: The final count.} In this section, we are interested in the maximum over $F\in \cY$ of the ratio 
\begin{equation*}
    \frac{|F|}{|\{t\text{-sets covered by an edge in }F\}|+|\{t\text{-sets claimed by diamonds in }F\}|},
\end{equation*}
which is an upper bound for $|G|/\tbinom{n}{t}$ since the edges of $G$ are partitioned into clusters of $\cY$ while every $t$-set is claimed or covered at most once. 

Clearly a $t$-tight component on one edge covers $\tbinom{r}{t}$ $t$-sets, and a component in $\cY$ on two edges covers or claims $\tbinom{2r-t}{t}$ $t$-sets. Concerning a component $\{e_1, e_2, e_3\}$ on three edges in $\cY$, suppose that $|e_1\cap e_2| = t$ and $|e_2\cap e_3| = t$. Since $G$ is $(3r-2t-1,3)$-free, we have that $e_1\cap e_3\subseteq e_2$. Then, either $|e_1\cap e_2\cap e_3| = t$, in which case there are $3\tbinom{2r-t}{t} - 3\tbinom{r}{t} + 1$ covered or claimed $t$-sets, or $|e_1\cap e_2\cap e_3| < t$, in which case $\{e_1, e_3\}$ is not a diamond and the total number of covered or claimed sets is $2\tbinom{2r-t}{t} - \tbinom{r}{t}$. Moreover, for all $r\ge 3$ and $t\in [2,r-1]$ we have
\begin{equation}\label{eq X}
\begin{split}
\binom{2r-t}{t} = \frac{1}{t!} \prod_{i=0}^{t-1} (2r-t-i)
&= \frac{1}{t!} (2r-2t+1)(2r-2t+2) \prod_{i=0}^{t-3} (2r-t-i)\\
&\ge \frac{1}{t!} (2r-2t+1)(2r-2t+2) \prod_{i=0}^{t-3} (r-i) \ge 2\binom{r}{t}.
\end{split}
\end{equation}
(Note that if $t=2$, then the empty products from 0 to 
$t-3$ are both agreed to be equal to~1.) This directly implies that
\begin{equation}\label{eq max}
    \max\left\{\frac{1}{\tbinom{r}{t}}, \frac{2}{\tbinom{2r-t}{t}}, \frac{3}{3\tbinom{2r-t}{t} - 3\tbinom{r}{t} + 1}, \frac{3}{2\tbinom{2r-t}{t} - \tbinom{r}{t}}\right\} = \binom{r}{t}^{-1}.
\end{equation}

Furthermore, a cluster consisting of an edge $e$ and $i$ diamonds, each claiming a different $t$-subset of $e$, claims or covers a total of $\tbinom{r}{t} + i\left(\tbinom{2r-t}{t}-1\right)$ $t$-sets. In this case, the ratio of edges to covered or claimed $t$-sets is
\begin{equation}\label{eq max 2}
    \frac{2i+1}{\tbinom{r}{t} + i\left(\tbinom{2r-t}{t}-1\right)} = \frac{2}{\tbinom{2r-t}{t}-1} + \frac{\tbinom{2r-t}{t}-2\tbinom{r}{t}-1}{\left(\tbinom{2r-t}{t}-1\right)\left(\tbinom{r}{t} + i\left(\tbinom{2r-t}{t}-1\right)\right)}.
\end{equation}
To understand which $i$ maximises the ratio, we look at the sign of the difference $\tbinom{2r-t}{t}-2\tbinom{r}{t}-1$. Note that~\eqref{eq X} holds with equality if and only if $t=2$ (ensuring the equality in the first inequality) and $r-t=1$ (ensuring equality in the second inequality). Hence, for $(t,r) \neq (2,3)$ we have that $\tbinom{2r-t}{t}\ge 2\tbinom{r}{t}+1$, so the maximum in~\eqref{eq max 2} is attained by the term $i=0$, while for $(t,r)=(2,3)$ the maximum is attained for $i = \tbinom{r}{t} = 3$. Note that~\eqref{eq max 2} for $i=0$ is equal to the right hand side of~\eqref{eq max}.

We conclude that for every cluster $F\in \cY$, if $(t,r)\neq (2,3)$, then the maximum ratio of edges in $F$ to $t$-sets covered by an edge in $F$ or claimed by $F$ is $\tbinom{r}{t}^{-1}$, while for $(t,r) = (2,3)$ this ratio is given by $7/18$. This finishes the proof of the upper bounds.


\section{Concluding remarks}\label{sec discussion}

In this paper, we made progress towards an old problem of Brown, Erd\H{o}s and S\'os. In particular, we settled the $(6,4)$-problem for $3$-graphs. In fact, our arguments can be adapted to obtain a limit of $\tfrac{1}{5}$ in each of the cases $k=5$ and $k=7$ of Conjecture~\ref{conj:BES}. Their proofs are omitted as the further insights needed for each of them do not quite adapt to the general case and mostly require a careful case analysis.


We also suggested a way to prove Conjecture~\ref{conj:BES} in general, namely by reducing it to Problem~\ref{prob:clean}. Shortly after this paper has been made available as a preprint, 
	Delcourt and Postle~\cite{DP:22} proved Conjecture~\ref{conj:BES} in full. A major ingredient in their work is our Corollary~\ref{cor:complete pack}. Interestingly, they found a way to prove the conjecture without solving Problem~\ref{prob:clean} directly. Instead, they show that in any sufficiently dense $(k+2,k)$-free $F$, one can find a subgraph $F'$ with the same density which is $(\ell+1,\ell)$-free for all $\ell\in [2,k-1]$. We refer the reader to their paper for more details. We also note that Shangguan~\cite{Sha22} subsequently extended the proof to higher uniformities. Now that Conjecture~\ref{conj:BES} is proven, it would be interesting to determine the limiting constants.\\
	
\vspace{0.5em}
\noindent
\textbf{Acknowledgements.} We are grateful to the two anonymous referees for useful remarks and suggestions.



\normalem

\bibliographystyle{plain}
\bibliography{Refs_BES}

\end{document}